\documentclass[a4paper,12pt]{amsart}
\usepackage{amsmath}
\usepackage{amssymb}
\usepackage{amsthm}
\usepackage{graphics}
\usepackage{hyperref}
\usepackage{tikz}
\usepackage{xcolor}
\usepackage{float}
\usetikzlibrary{arrows}
\usepackage{ragged2e}
\usepackage{calculator}
\usepackage{soul}

\usepackage[pagewise]{lineno}%\linenumbers

\newtheorem{theorem}{Theorem}[section]
\newtheorem{definition}[theorem]{Definition}
\newtheorem{proposition}[theorem]{Proposition}
\newtheorem{lemma}[theorem]{Lemma}

\newtheorem{remark}[theorem]{Remark}

\renewcommand{\Re}{\operatorname{Re}}

\newcommand{\N}{\mathbb{N}}

\newcommand{\norm}[1]{{\left\| {#1} \right\|}}
\newcommand{\pare}[1]{\left( #1\right)}

\newcommand{\vvector}[3]{\begin{bmatrix}
    #1\\#2\\#3
\end{bmatrix}  }
\newcommand{\operA}{\mathbb{A}}

\numberwithin{equation}{section}

\title[Semilinear Moore-Gibson-Thompson equations]{Local well-posedness for a class of semilinear Moore-Gibson-Thompson equations with subcritical nonlinearities} 

\author[F. D. M. Bezerra]{Flank D. M. Bezerra$^*$}\thanks{$^*$Research partially supported by
CNPq \# 303039/2021-3, Brazil.}
\address[F. D. M. Bezerra]{Departamento de Matem\'atica, Universidade Federal da Para\'iba, 58051-900 Jo\~ao Pessoa PB, Brazil.}
\email{flank@mat.ufpb.br}

\author[L. M. Salge]{Luís M. Salge}
\address[L. M. Salge]{Instituto de Matemática e Estatística, Universidade do Estado do Rio de Janeiro, 20550-900 Rio de Janeiro RJ, Brazil.}
\email{luis.salge@ime.uerj.br}
\date{\today}

\begin{document}

\begin{abstract}
In this paper, we study a class of higher-order semilinear evolution equations inspired by the Moore-Gibson-Thompson model introduced by Dell'Oro, Liverani and  Pata (2023), involving strongly elliptic operators of order ($2m$) with homogeneous boundary conditions. The associated unbounded linear operator is   a sectorial operator with zero belonging to the resolvent set, allowing the construction of fractional powers spaces, and the analysis of their spectral properties. We prove the local well-posedness of the corresponding semilinear Cauchy problem under subcritical nonlinearities. Our framework clarifies the role of extrapolation spaces and fractional domains in handling the lack of accretivity and bounded imaginary powers of the operator.

\end{abstract}

\maketitle

\tableofcontents

\section{Introduction}

Let $\Omega\subset\mathbb{R}^N$ be a bounded domain of class $\mathcal{C}^{2m+\mu}$, where $N>2m$, $m\in\mathbb{N}$ and $\mu\in(0,1)$. For  a scalar function $u=u(x,t)$ defined on $\Omega\times(0,T)$ we denote by $D^\alpha u$ its partial derivative $D_1^{\alpha_1}\ldots D^{\alpha_n}_nu$, where $D_i= \partial_{x_{i}}$, $\alpha=(\alpha_1,\ldots,\alpha_n)$ is  a multi-indice and $|\alpha|=|\alpha_1|+\ldots+|\alpha_n|$. In this paper, $\mathcal{A}$ denotes the differential operator given by
\[
\mathcal{A}u=\displaystyle\sum_{|\alpha|,|\beta|\leq m} (-1)^{|\beta|}D^\beta(a_{\alpha,\beta}(x)D^\alpha u),
\]
which is a $2m$-th order strongly elliptic differential operator. All coefficients $a_{\alpha,\beta}$ are of class  $C^{4m+|\beta|+\mu}(\Omega)$ and bounded in $\Omega$ by a constant $c_{up}<0$.

Together with \eqref{Eq01Local}, we consider the boundary condition
\begin{equation}\label{Condd01}
B_0u(x,t)=B_1u(x,t)=\ldots=B_{m-1}u(x,t)=0,\quad  x\in\partial\Omega,\ t>0,
\end{equation}
where $B_0,B_1,\ldots,B_{m-1}$ are linear and time independent boundary operators, i.e.,
\[
B_j=\displaystyle\sum_{|\sigma|\leq m_j}b_\sigma^j(x)D^\sigma,\ j=1,\ldots,m,\ 0\leq m_j\leq 2m-1.
\]

The triple $(-\mathcal{A},\{B_j\},\Omega)$ forms a ‘‘regular elliptic boundary value problem'', according to \cite[Pages 1-3]{Ch}, \cite[Definition 1.2.1]{ChD}, \cite[Pages 1-3]{Ch0}, \cite[Page 76]{AF}, \cite[Page 125]{LM} and \cite{JN}. According to \cite{AF}, it satisfies the ‘‘root condition'', the ‘‘smoothness condition'', certain ‘‘complementary condition'' and the system $\{B_j\}$ is normal. In particular, we need our differential operator $A$ to be uniformly strongly elliptic in $\Omega$; that is, we assume that for some $C>0$
\[
\displaystyle\sum_{|\alpha|,|\beta|=m}a_{\alpha,\beta}(x)\xi^\alpha\xi^\beta\geq C|\xi|^{2m},
\]
for any $ x\in\Omega$, $\xi\in\mathbb{R}^n$.

Then  the unbounded linear operator $(\mathcal{A},D(\mathcal{A}))$ acting in $L^2(\Omega)$  with domain  
\[
D(\mathcal{A})=\{u\in H^{2m}(\Omega); B_0u=\ldots=B_{m-1}u=0\ \mbox{on}\ \partial\Omega\},
\]
which is the close subspace of the space $H^{2m}(\Omega)$ consisting of functions satisfying, in the sense of traces, boundary conditions $B_j$.  The bilinear form $p_{\mathcal{A}}$ connected to the linear operator $\mathcal{A}$, given by
\[
p_{\mathcal{A}}(u,v)=\displaystyle\sum_{|\alpha|,|\beta|\leq m} a_{\alpha,\beta}(x) D^\alpha u D^\beta v,
\]
satisfies the following conditions:

\noindent i) Coerciveness inequality
\begin{equation}\label{coerciveness}
\int_\Omega p_{\mathcal{A}}(w,w)dx+C_2\|w\|_{L^2(\Omega)}^2\geq C_1\|w\|^2_{H^{m}(\Omega)}
\end{equation}
for any $w\in H^{m}_0(\Omega)$. 

\noindent ii) Green's Identity
\[
\int_\Omega (\mathcal{A}w)wdx=\int_\Omega p_{\mathcal{A}}(w,w)dx.
\]
for any $w\in H^{m}_0(\Omega)$. 

Therefore, $(\mathcal{A}+\omega I,D(\mathcal{A}+\omega I))$ with domain $D(\mathcal{A}+\omega I)=D(\mathcal{A})$  is a sectorial operator in $L^2(\Omega)$ for each $\omega\in\mathbb{R}$, according to \cite[Example 1.3.8]{ChD}. In particular,  $\mu_0\geq0$ may be chosen in \eqref{Eq01Local}, for which  $\mathcal{A}+\mu_0I$ becomes positive in $L^2(\Omega)$, consequently the spectrum of $\mathcal{A}+\mu_0I$ is a point spectrum consisting of eigenvalues $0<\lambda_0\leq\lambda_1\leq\cdots\leq\lambda_{n-1}\leq\lambda_n\leq\cdots$, with $\lim_{n\to\infty}\lambda_n=\infty$, where $\{\lambda_n\}_n$ denotes the ordered sequence of eigenvalues of $\mathcal{A}+\mu_0I$,  including their multiplicity.

\subsection{Main problem}

From now on, we denote
\[
A=\mathcal{A}+\mu_0I,
\]
and inspired by the class of Moore-Gibson-Thompson (MGT) equations  treated in \cite{DLP,KaLaMa} and \cite{KaLaPo}, we consider the semilinear  evolution 
 equation with $2m$-th order ($m>1$) elliptic operator
\begin{equation}\label{Eq01Local}
\partial_t^3u+\alpha \partial_t^2u+\beta A\partial_tu+\gamma Au +\delta A\partial_t^2u=f(u,\partial_tu,\partial_t^2u),
\end{equation}
subject to the boundary condition \eqref{Condd01}, and initial condition
\begin{equation}\label{CondInCondi}
    u(x,0)=u_0(x),\quad x\in\Omega,
\end{equation}
where $\alpha,\beta,\gamma,\delta$ are positive constants, with $x\in\Omega$, $t>0$, $u=u(x,t)$, with the following condition
\[
\dfrac{\gamma}{\alpha+\delta\lambda_0}<\beta,
\]
as well as in \cite[Theorem 5.1 and Remark 5.2, Theorem 4.1]{DLP}.

 It is often convenient to write the MGT equation in abstract form \eqref{Eq01Local}  and it has been shown \cite{KaLaMa,MaMcDeTri,HoZhaHuJi}.

Let $\sigma$ be an arbitrary real number. We introduce the  $\sigma-$th power of  the operator $A$, acting in $X$ according to the rule 
\[
A^\sigma \phi=\displaystyle\sum_{k=1}^\infty \lambda_k^\sigma\phi_ke_k,
\]
where $\phi_k$ is the Fourier coefficients of a function $\phi\in X$, $\phi_k=(\phi,e_k)$, see \cite[Section 6.4]{CarLanRob}. The domain of this
operator has the form
\[
D(A^\sigma)=\left\{\phi\in H;\displaystyle\sum_{k=1}^\infty \lambda_k^{2\sigma}(\phi, e_k)^2<\infty\right\}.
\]
It immediately follows from this definition that $D(A^{\sigma_1})\subset D(A^{\sigma_2})$ for any $\sigma_1\geq \sigma_2$. On the set $D(A^{\sigma})$, we define the inner product
\[
(\phi,\varphi)_\sigma=\displaystyle\sum_{k=1}^\infty \lambda_k^{2\sigma}\phi_k\varphi_k=(A^\sigma\phi,A^\sigma\varphi).
\]
For element of $D(A^\sigma)$ we introduce the norm
\[
\|\phi\|_{D(A^\sigma)}^2=(\phi,\phi)_\sigma=\displaystyle\sum_{k=1}^\infty \lambda_k^{2\sigma}|\phi_k|^2=(A^\sigma\phi,A^\sigma\phi)=\|A^\sigma\phi\|.
\]

Let $X^\alpha$ be the fractional power spaces associated to the operator $A$, $0\leq\alpha\leq1$; namely, $X^\alpha=D(A^\alpha)$ the domain of the linear operator $A^\alpha$ endowed with the graphic norm. In particular
\[
X^0=L^2(\Omega),
\]
\[
X^{\frac12}=
\begin{cases}
H^{m}_0(\Omega),& \mbox{if}\ B_0=\mbox{Identity},\\
H^{m}(\Omega),& \mbox{other cases},
\end{cases}
\]
and
\[
X^1=D(A)=\{u\in H^{2m}(\Omega); B_0u=\ldots=B_{m-1}u=0\ \mbox{on}\ \partial\Omega\}.
\]

Recall that the extrapolation space $X_{(-1)}$ of $X$ generated by $A$ is the completion of the normed space $(X,\|A^{-1}\cdot\|_X)$. Let $A_{(-1)}$ be the closure of $A$ in $X_{(-1)}$, we have
\[
X_{(-1)}=H^{-m}(\Omega)
\]
and
\[
D(A_{(-1)})=
\begin{cases}
H^{m}_0(\Omega),& \mbox{if}\ B_0=\mbox{Identity},\\
H^{m}(\Omega),& \mbox{other cases},
\end{cases}
\]

One way to derive \eqref{Eq01Local}, for the concrete choice of the linear operator $A$ equal to negative Laplacian operator with zero Dirichlet boundary condition, $-\Delta_D$, on the boundary of a bounded smooth domain $\Omega\subset\mathbb{R}^N$ with $N>2cm$,  is to view it as a model describing heat conduction, see e.g. \cite{DLP,KaLaMa,KaLaPo}. In this case, we have
\begin{equation}\label{FreFtgyy}
-\Delta_D:D(-\Delta_D):=H^2(\Omega)\cap H_0^1(\Omega)\subset L^2(\Omega)\to L^2(\Omega)\quad     
u\mapsto -\Delta_D u=-\displaystyle\sum_{i=1}^N\partial_{x_i}^2u.
\end{equation}

In terms of nonlinearity  in \eqref{Eq01Local}, to make the presentation easier, let us assume that $f$ in \eqref{Eq01Local} is such that
\begin{equation}\label{CondNonLinf}
f(u,\partial_tu,\partial_t^2u)=f_1(u)+f_2(\partial_tu)+f_3(\partial_t^2u),
\end{equation}
where $f_\ell:\mathbb{R}\to\mathbb{R}$ is a twice continuously differentiable function for any $\ell=1,2,3$.  We shall require $f_i$ fulfill the estimates
\begin{equation}\label{NovaEstGGDDDima}
|f_i''(s)|\leq c(1+|s|^{\varrho-2})    
\end{equation}
for any $s\in\mathbb{R}$, with $1<\varrho\leq\frac{N+2m}{N-2m}$, $i=1,2$,  for some $c>0$, and $f_3$ being globally Lipschitz continuous.

\subsection{Definitions and Notations}

Throughout this paper, we use the following notations 
\[
Y=X^{\frac12}\times X^{\frac12}\times X
\]
endowed with the standard product norm
\[
\left\|\begin{bmatrix}
    u\\ v\\ w
\end{bmatrix}\right\|_{Y}^2=\|u\|^2_{X^{\frac12}}+\|v\|^2_{X^{\frac12}}+\|w\|^2_X
\]
and we denote by $L(Y)$ the Banach space of bounded linear operators on $Y$.

Introducing the state vector
\[
{\bf u}(t)=\begin{bmatrix}u(t)\\ v(t)\\ w(t)\end{bmatrix}
\]
we view \eqref{Eq01Local} as the semilinear ODE in $Y$
\begin{equation}\label{SemilinearCauchyProb}
\dfrac{d{\bf u}}{dt}+\mathbb{A}{\bf u}=\mathbb{F}({\bf u}).
\end{equation}
Here, $\mathbb{A}$ is the unbounded linear operator on $Y$ acting as
\begin{equation}\label{OpA001}
\forall \begin{bmatrix}     u\\ v\\ w \end{bmatrix}\in  D(\mathbb{A}),\quad \mathbb{A}\begin{bmatrix}     u\\ v\\ w \end{bmatrix}=\begin{bmatrix}v\\ w\\ -\alpha w-A(\beta v+\gamma u+\delta w)\end{bmatrix}
\end{equation}
of dense domain
\begin{equation}\label{OpA002}
D(\mathbb{A})=\left\{\begin{bmatrix}     u\\ v\\ w \end{bmatrix}\in Y; w\in X^{\frac12},\beta v+\gamma u+\delta w\in X^1\right\}.    
\end{equation}

The nonlinearity $F$ is given by
\begin{equation}\label{NonLinF}
\forall \begin{bmatrix}     u\\ v\\ w \end{bmatrix}\in Y,\quad \mathbb{F}({\bf u})=\begin{bmatrix}
   0\\ 0\\ f(u,v,w) 
\end{bmatrix}.
\end{equation}

As we shall see, zero belongs to the resolvent set of the closed and densely defined operator $\mathbb{A}$, and if we denote the extrapolation space $Y_{(-1)}$ of $Y$ generated by $\mathbb{A}$, then it is the completion of the normed space $(Y,\|\mathbb{A}^{-1}\cdot\|_Y)$. Some of the difficulties to be addressed here are:
\begin{itemize}
\item[$i)$] The domain of the linear operator $\mathbb{A}_{(-1)}$ in $Y_{(-1)}$, closure of the linear operator $\mathbb{A}$ is a cross product, as well as, it is the domain the linear operator $\mathbb{A}$ in $Y$. This is not the case when $\mathbb{A}$ is the strongly damped wave operator as was observed in the papers \cite{carab,CC0,CC,CC1,PaMa} and \cite{ViPaMaSqua}; 
\item[$ii)$] The linear operator $\mathbb{A}$ does not have bounded imaginary power in $Y$;
\item[$iii)$] The linear operator $\mathbb{A}$ is not accretive in $Y$.
\end{itemize}

The structure of the paper is as follows. In Section \ref{S2}  we recall fundamental aspects of the theory of fractional powers of non-negative  (and not necessarily self-adjoint) operators, which have important links with partial differential equations.  Moreover, we study the spectral properties of the fractional power operators $\mathbb{A}$.  Finally, in Section \ref{MainSection} we prove the main result of the paper, and present a result of local well-posedness for the semilinear Cauchy problem associated with the differential equation \eqref{SemilinearCauchyProb}.

\section{Spectral properties of the linear operators}\label{S2}

Our objective in this section is to prove the global existence of solutions and the existence of a pullback attractor for the boundary-initial-value problem 
\begin{equation}\label{MainPro}
\begin{cases}
\dfrac{d{\bf u}}{dt}+\mathbb{A}{\bf u}=\mathbb{F}({\bf u}),\ t>0,\\
{\bf u}(0)={\bf u}_0,
\end{cases}
\end{equation}
To do this, consider the abstract semilinear Cauchy problem

\begin{equation}\label{abstrcauchyp}
\begin{cases}
\dfrac{dz}{dt} +Bz = F(z), \ t > 0,\\
z(0) = z_0.
\end{cases}
\end{equation}

\begin{definition}
Let $Z$ be a Banach space, and let $B:D(B)\subset Z\to Z$ be a sectorial operator.   For a continuous function $F:Z^{\eta} \to Z$, with $\eta \in [0,1)$, $z(\cdot, z_0): [0, T_0) \to Z^{\eta}$ is a solution for \eqref{abstrcauchyp} if it is continuous, continuously differentiable in $(0, T_0), z(t,  z_0) \in D(B)$ for $t \in (0, T_0)$ and \eqref{abstrcauchyp} is satisfied for all $t \in (0, T_0)$.
\end{definition}

\begin{theorem}\label{existenceofsolution}
    If the linear operator B is  sectorial, and $F : Z^{\eta} \to Z$ is Lipschitz continuous in bounded subsets of $Z^\eta$ with $\eta\in[0,1)$, then, given $r > 0$ there exists a $T_0 > 0$ and for each $z_0 \in Z^{\eta}$ with $\|z_0\|_{Z^{\eta}} \leq r$ a function $z(\cdot,  z_0) \in C([0, T_0], Z^\eta) \cap C^{1}((0, T_0], Z^{\eta})$  with the property that
    $$\{z_0 \in Z^\eta: \|z_0\|_{Z^{\eta}}\ \leq r\} \ni z_0 \mapsto z(\cdot,  z_0) \in C([0,  T_0], Z^\eta)$$
is continuous; $ z(\cdot, z_0)$ is the unique solution of \eqref{abstrcauchyp}.
\end{theorem}

\begin{theorem}
Let $\mathbb{A}$ be the unbounded linear operator defined by \eqref{OpA001}-\eqref{OpA002}. The following conditions hold:
\begin{itemize}
\item[$i)$] $-\mathbb{A}$ is  the infinitesimal generator of a strongly continuous semigroup in $Y$. The strongly continuous semigroup generated by  $-\mathbb{A}$ is exponentially stable if and only if $\chi>0$;
\item[$ii)$] Zero belongs to the resolvent set of $\mathbb{A}$;  namely the bounded linear operator  $\mathbb{A}^{-1}:Y\to Y$ is defined by
\[
\mathbb{A}^{-1}= \begin{bmatrix}
    -\gamma^{-1} \beta I & -\gamma^{-1} \delta I -\gamma^{-1} \alpha A^{-1} & -\gamma^{-1} A^{-1}  \\
    I & 0 & 0 \\
    0 & I & 0
\end{bmatrix}
\]
\item[$iii)$] $\mathbb{A}$ is a sectorial operator in $Y$.
\end{itemize}
\end{theorem}

\begin{proof}
The item $i)$ is proved in \cite[Remark 2.3, Theorem 2.1 and Theorem 5.1]{DLP}. The item $ii)$ is not difficult to see. Finally, the item $iii)$ is proved in \cite[Theorem 4.1]{DLP}. 
\end{proof}

\begin{remark}
Let $\mathbb{A}$ be the unbounded linear operator defined by \eqref{OpA001}-\eqref{OpA002}. The operator \st{$\mathbb{A}$ but}  $\mathbb{A}$ does not have  compact resolvent. Indeed, if $\{u_{n}\}_{n} \subset X^{\frac{1}{2}}$ is a bounded  sequence that does not have convergent subsequence. Then 
\[
\operA^{-1}\begin{bmatrix}u_{n}\\ 0\\ 0\end{bmatrix} = \begin{bmatrix}-\gamma^{-1}\beta u_{n}\\ u_{n}\\ 0\end{bmatrix}
\]
does not have a convergent subsequence.
\end{remark}

\begin{remark}
We draw attention to the fact that $\mathbb{A}$ is not accretive in $Y$. Indeed, let $(0,0,w) \in D(\operA)$ with $w \neq 0,$ then 
\[
\begin{split}
\left(\operA \begin{bmatrix}0\\ 0\\ w\end{bmatrix},\begin{bmatrix}0\\ 0\\ w\end{bmatrix}\right)_{Y} &= ( -\alpha w - \delta A w, w)_{X}\\
&= -\alpha \norm{w}^{2}_{X} -\delta (Aw,w)_{X}
\end{split}
\]
that is,
\[
\begin{split}
\Re(\operA (0,0,w),(0,0,w))_{Y} = -\alpha \norm{w}^{2}_{X} -\delta (Aw,w)_{X} \leq -\alpha \norm{w}^{2}_{X} <0. 
\end{split}
\]
\end{remark}

Recall that the extrapolation space $Y_{(-1)}$ of $Y$ generated by $\mathbb{A}$ is the completion of the normed space $(Y,\|\mathbb{A}^{-1}\cdot\|_Y)$. We may infer that:

\begin{theorem}\label{TSerFRef}
Let $\mathbb{A}_{(-1)}$ be the closure of $\mathbb{A}$ in $Y_{(-1)}$. Then
\begin{itemize}
\item[$i)$] $\mathbb{A}_{(-1)}$ is sectorial  operator in the space $Y_{(-1)}$  with 
\begin{equation}\label{domaindef}
\begin{split}
D(\operA_{(-1)}) = &\left\{ \vvector{u}{v}{w} \in X^{\frac12}\times X^{\frac12}\times X^{\frac12}; \exists \vvector{u_{n}}{v_{n}}{w_{n}}_{n\in \N}\subset D(\operA), 
\vvector{u_{n}}{v_{n}}{w_{n}} \overset{X^{\frac12}\times X^{\frac12} \times X^{\frac12}}{\to}\vvector{u}{v}{w},\right. \\
&\left.
%\beta v_{n}+\gamma u_{n} +\delta w_{n} \in D(A) 
\mbox{ and } A(\beta v_{n}+\gamma u_{n} +\delta w_{n}) \overset{X^{-\frac12}}{\to} A_{(-1)}(\beta v+\gamma u +\delta w) \right\};
\end{split}.
\end{equation}
\item[$ii)$] $Re\sigma(\mathbb{A})>0$;
\item[$iii)$] For any $\alpha>0$ we can define the bounded linear operator on $Y$ by
\[
\mathbb{A}^{-\alpha}=\dfrac{1}{\Gamma(\alpha)}\int_0^\infty t^{\alpha-1}e^{-\mathbb{A}t}dt,
\]
which is one-one, where $\Gamma(\alpha)$ denotes the gamma function. For each $\alpha>0,\beta>0$, $\mathbb{A}^{-\alpha}\mathbb{A}^{-\beta}=\mathbb{A}^{-(\alpha+\beta)}$, and moreover, we can define $\mathbb{A}^{\alpha}$ being inverse of $\mathbb{A}^{-\alpha}$ ($\alpha>0$), $Y^\alpha=D(\mathbb{A}^{\alpha})=R(\mathbb{A}^{-\alpha})$, $\mathbb{A}^0=$ identity on $Y$.  
\item[$iv)$] $\mathbb{A}_{(-1)}$ has compact resolvent.   
\end{itemize}
\end{theorem}

\begin{proof}
Firstly note that  
\[
\begin{split}
D(\operA_{(-1)}) = &\left\{ \vvector{u}{v}{w} \in X^{\frac12}\times X^{\frac12}\times X^{\frac12}; \exists \vvector{u_{n}}{v_{n}}{w_{n}}_{n\in \N}\subset D(\operA), 
\vvector{u_{n}}{v_{n}}{w_{n}} \overset{X^{\frac12}\times X^{\frac12} \times X^{\frac12}}{\to}\vvector{u}{v}{w},\right. \\
&\left.
%\beta v_{n}+\gamma u_{n} +\delta w_{n} \in D(A) 
\mbox{ and } A(\beta v_{n}+\gamma u_{n} +\delta w_{n}) \overset{X^{-\frac12}}{\to} A_{(-1)}(\beta v+\gamma u +\delta w) \right\} 
\end{split}.
\]
 Indeed,  $\begin{bmatrix}     u\\ v\\ w \end{bmatrix} \in D\pare{\operA_{(-1)}}$ if, and only if, there exists $\left\{ \vvector{u_{n}}{v_{n}}{w_{n}}\right\}_{n \in \N}\subset D(\mathbb{A})$ such that 
\begin{align}
\norm{\vvector{u}{v}{w}-\vvector{u_n}{v_n}{w_n}}_{Y_{(-1)}} \to 0 \\
\norm{\operA_{(-1)}\vvector{u}{v}{w}-\operA\vvector{u_n}{v_n}{w_n}}_{Y_{(-1)}} \to 0.
\label{convdef}
\end{align}
Denoting $\vvector{\mathcal{U}}{\mathcal{V}}{\mathcal{U}} \doteq \operA_{(-1)}\vvector{u}{v}{w},$ we have that $u_n \overset{X^{\frac12}}{\to} u, v_n \overset{X^{\frac12}}{\to} v, 
$ $w_n \overset{X^{-\frac12}}{\to} w$ and $v_n \overset{X^{\frac12}}{\to} \mathcal{U}, w_n \overset{X^{\frac12}}{\to} \mathcal{V}, $ $A(\beta v_n +\gamma u_n + \delta w_n) \overset{X^{-\frac12}}{\to} -\mathcal{W} -\alpha w.$
Therefore $\mathcal{U} = v, \mathcal{V} = w,$ $\beta v_n + \gamma u_n + \delta w_n \in D(A)$ with $\beta v_n + \gamma u_n + \delta w_n \overset{X^{\frac12}}{\to} \beta v + \gamma u + \delta w$ and $A(\beta v_n + \gamma u_n + \delta w_n) \overset{X^{-1}}{\to} -\mathcal{W} - \alpha w,$ i.e., $\beta v + \gamma u + \delta w \in D\pare{A_{(-1)}}$ with $A_{(-1)} \pare{\beta v + \gamma u + \delta w} = -\mathcal{W} - \alpha w$ hence 
$$
\operA_{(-1)}\vvector{u}{v}{w} = \vvector{\mathcal{U}}{\mathcal{V}}{\mathcal{W}} = \vvector{v}{w}{-\alpha w - A_{(-1)}\pare{\beta v + \gamma u + \delta w}}.
$$

On the other hand, let $\vvector{u}{v}{w} \in X^{\frac12}\times X^{\frac12}\times X^{\frac12}$ such that there exists $\vvector{u_{n}}{v_{n}}{w_{n}}_{n\in \N}\subset X^{\frac12}\times X^{\frac12} \times X^{\frac12}$ with $\vvector{u_{n}}{v_{n}}{w_{n}} \overset{X^{\frac12}\times X^{\frac12} \times X^{\frac12}}{\to}\vvector{u}{v}{w},$ $\beta v_{n}+\gamma u_{n} + \delta w_{n} \in D(A)$ and $A(\beta v_{n}+\gamma u_{n} + \delta w_{n}) \overset{X^{-\frac12}}{\to} A_{(-1)}(\beta v+\gamma u + \delta w).$ Then $\vvector{u_{n}}{v_{n}}{w_{n}} \overset{Y_{(-1)}}{\to}\vvector{u}{v}{w} $ and $$\operA\vvector{u_{n}}{v_{n}}{w_{n}} = \vvector{v_{n}}{w_{n}}{-\alpha w_{n} -A(\beta v_n +\gamma u_n +\delta w_n)} \overset{Y_{(-1)}}{\to}\vvector{v}{w}{-\alpha w-A_{(-1)}(\beta v + \gamma u + \delta w)} =\operA_{(-1)}\vvector{u}{v}{w},$$
i.e., $\vvector{u}{v}{w} \in D(\operA_{(-1)}).$

Thanks to \cite[Pg. 262]{A} we have $\rho(\operA) = \rho(\operA_{(-1)})$ and, since $Y^{1}$ is dense in $Y,$ $\operA_{(-1)}$ is a continuous extension of $\operA$ onto $Y$ such that $\operA_{(-1)}$ is an isometric isomorphism from $Y$ onto $Y_{(-1)}.$ 
Therefore, for each $\begin{bmatrix}     u\\ v\\ w \end{bmatrix} \in Y$ 
\[
\begin{split}
\norm{(\lambda - \operA_{(-1)})^{-1} \begin{bmatrix}
           u \\
           v \\
           w
         \end{bmatrix}}_{Y_{(-1)}} &= \norm{\operA^{-1}(\lambda - \operA_{(-1)})^{-1} \begin{bmatrix}
           u \\
           v \\
           w
         \end{bmatrix}}_{Y},
\end{split}
\]
and consequently
\[
\begin{split}
\norm{\operA^{-1}(\lambda - \operA)^{-1} \begin{bmatrix}
           u \\
           v \\
           w
         \end{bmatrix}}_{Y}  & = \norm{(\lambda - \operA)^{-1}\operA^{-1} \begin{bmatrix}
           u \\
           v \\
           w
         \end{bmatrix}}_{Y}\\
         &\leq \frac C\lambda \norm{\operA^{-1} \begin{bmatrix}
           u \\
           v \\
           w
         \end{bmatrix}}_{Y}  \\
         &\leq \frac C\lambda \norm{ \begin{bmatrix}
           u \\
           v \\
           w
         \end{bmatrix}}_{Y_{-1}},
\end{split}
\]
    where $|\lambda|\norm{(\lambda -\operA)^{-1}} \leq \frac C \lambda, \lambda \in \rho(\operA),$ because $\operA$ is sectorial.

    The above inequality can be extended to the elements of $Y_{(-1)}$ by density, thus $\operA_{(-1)}$ is sectorial. 

The remaining assertion is the direct consequence of the results reported  in \cite[Lemma 1.3.7 and  Theorem 1.3.8]{A} and 
\cite[Definition 1.4.1, Theorem 1.4.2, Definition 1.4.1, {\textit{continued}}]{H}. The proof is complete. $\Box$
\end{proof}

\begin{remark}
Note that the domain of the linear operator $\mathbb{A}_{(-1)}$ in $Y_{(-1)}$ is a cross product, as well as, it is the domain the linear operator $\mathbb{A}$ in $Y$. This not is the case when $\mathbb{A}$ is the strongly damped wave operator as was observed in the papers \cite{carab,CC0,CC,CC1,ViPaMaSqua} and \cite{PaMa}. 
\end{remark}

We shall next study \eqref{MainPro} as a sectorial problem in $Y_{(-1)}$. 

\begin{proposition}
Let $Y_{(-1)}$ denote the extrapolation space of $Y$ generated by $\mathbb{A}$. The following equality holds:
\begin{equation}\label{AADsweF}
Y_{(-1)}=X^{\frac{1}{2}}\times X^{\frac 12} \times X^{-\frac 12}.
\end{equation}
\end{proposition}

\begin{proof}
Recall first that $Y_{(-1)}$ is the completion of the normed space $(Y,\|\mathbb{A}^{-1}\cdot\|_Y)$. 
Note that
\begin{align*}
    &\norm{\gamma^{-1}[\beta u +\delta v+ A^{-1}(\alpha v +w)]}_{X^{\frac{1}{2}}} \leq \\
    &\leq |\gamma|^{-1}\left[|\beta|\norm{u}_{X^{\frac{1}{2}}}+(|\delta| +
    \norm{A^{-1}}_{\mathcal{L}(X)}|\alpha|)\norm{v}_{X^{\frac{1}{2}}} + \norm{w}_{X^{-\frac 12}}\right],
\end{align*}
since
\begin{align*}
    \norm{\operA^{-1}\begin{bmatrix}     u\\ v\\ w \end{bmatrix}}^{2}_{Y}  = \norm{\gamma^{-1}[\beta u +\delta v +A^{-1}(\alpha v +w)]}^{2}_{X^{\frac{1}{2}}} + \norm{u}^{2}_{X^{\frac{1}{2}}} +\norm{v}^{2}_{X} \leq \\
    \leq \norm{\gamma^{-1}[\beta u +\delta v +A^{-1}(\alpha v +w)]}^{2}_{X^{\frac{1}{2}}} + \norm{u}^{2}_{X^{\frac{1}{2}}} +\norm{v}^{2}_{X^{\frac 12}}
\end{align*}
then
\begin{align*}
    \norm{\operA^{-1}\begin{bmatrix}     u\\ v\\ w \end{bmatrix}}_{Y} \leq C \norm{\begin{bmatrix}     u\\ v\\ w \end{bmatrix}}_{X^{\frac{1}{2}} \times X^{\frac 12} \times X^{-\frac 12}}.
\end{align*}

On the other hand, note that 
\begin{align*}
    \norm{w}_{X^{-\frac 12}} &= \norm{A^{\frac 12}A^{-1}w}_{X}  \leq \\
    & \leq \norm{\gamma^{-1}[\beta u +\delta v+ A^{-1}(\alpha v +w)]}_{X^{\frac{1}{2}}} + \norm{\gamma^{-1}[\beta u + \delta v + A^{-1}\alpha v]}_{X^{\frac{1}{2}}}  \\
    &\leq
    \norm{\gamma^{-1}[\beta u +\delta v+ A^{-1}(\alpha v +w)]}_{X^{\frac{1}{2}}} + c_{1}\norm{u}_{X^{\frac{1}{2}}} + c_{2}\norm{v}_{X^{\frac 12}},
\end{align*}
where $c_1$ and $c_2$ are constants.
Hence, given $\begin{bmatrix}     u\\ v\\ w \end{bmatrix} \in Y$
\begin{align*}
    \norm{\begin{bmatrix}     u\\ v\\ w \end{bmatrix}}_{X^{\frac{1}{2}}\times X^{\frac 12} \times X^{-\frac 12}} \leq  c^\prime \norm{\operA^{-1}\begin{bmatrix}     u\\ v\\ w \end{bmatrix}}_{Y}.
\end{align*}

Since
\[
\left\|\mathbb{A}^{-1}\begin{bmatrix}     u\\ v\\ w \end{bmatrix}\right\|_Y\leq C\left\|\begin{bmatrix}     u\\ v\\ w \end{bmatrix}\right\|_{X^{\frac{1}{2}}\times X^{\frac 12} \times X^{-\frac 12}}\leq  C^\prime\left\|\mathbb{A}^{-1}\begin{bmatrix}     u\\ v\\ w \end{bmatrix}\right\|_Y
,\]
where $C^\prime = Cc^\prime,$
for any $\begin{bmatrix}     u\\ v\\ w \end{bmatrix}\in Y$, the completions $(Y,\|\mathbb{A}^{-1}\cdot\|_Y)$ and $\left(Y,\|\cdot\|_{X^{\frac 12}\times X^{\frac 12}\times X^{-\frac 12}}\right)$ coincide. Recalling that $Y=X^{\frac12}\times X^{\frac12}\times X$, we obtain \eqref{AADsweF}. The proof is complete. $\Box$
\end{proof}

Denote by $Y^1_{(-1)}$ the domain $D(\mathbb{A}_{(-1)})$ of the unbounded linear operator $\mathbb{A}_{(-1)}$ given by $\eqref{domaindef}$
%; namely, 
%\begin{equation}
%D(\mathbb{A}_{(-1)})=\left\{\begin{bmatrix}     u\\ v\\ w \end{bmatrix}\in Y_{(-1)}; w\in X^{\frac12},\beta v+\gamma u+\delta w\in D(A_{(-1)})\right\}.
%\end{equation}
with the norm $\|\mathbb{A}\cdot \|_{Y_{(-1)}}$, where
\begin{equation}
\forall \begin{bmatrix}     u\\ v\\ w \end{bmatrix}\in  D(\mathbb{A}_{(-1)}),\quad \mathbb{A}_{(-1)}\begin{bmatrix}     u\\ v\\ w \end{bmatrix}=\begin{bmatrix}v\\ w\\ -\alpha w-A_{-1}(\beta v+\gamma u+\delta w)\end{bmatrix}
\end{equation}

Denote by $Y_{(-1)}^\alpha$ the domain of $\mathbb{A}^\alpha_{(-1)}$ for some $\alpha\geq0$ with graph norm.  We obtain the fractional power scale spaces associated with $\mathbb{A}^\alpha_{(-1)}$, see \cite{H}.

The Figure \ref{LaStFiGure} illustrates some of the information we now have.
\begin{figure}[!htp]
\centering
\begin{tikzpicture}
\draw[-latex] (-4.5,1.2)node[left]{$D(\mathbb{A})$} -- (0.95,1.2)node[right]{$Y=X^{\frac12}\times X^{\frac12}\times X$};
\draw[-latex] (-2,-0.5)node[left]{$Y_{(-1)}^1=D(\mathbb{A}_{(-1)})$} -- (-1.5,-0.5)node[right]{$Y_{(-1)}^\alpha$};
\draw[-latex] (-0.4,-0.5) -- (0.5,-0.5)node[right]{$Y_{(-1)}^{\alpha_0}$};
\draw[-latex] (1.6,-0.5)-- (4,-0.5)node[right]{$Y_{(-1)}=X^\frac12\times X^{\frac{1}{2}} \times X^{-\frac 12}$};
%\draw[stealth-stealth] (2.5,0.8)--(2.5,-0.2)node[midway,fill=white]{Isomorphisms};
\end{tikzpicture}
\caption{Partial description of the fractional power spaces scales}\label{LaStFiGure}
\end{figure}

\section{Local well posedness with subcritical nonlinearities}\label{MainSection}

Finally, in this section, we prove a local solvability result for the problem \eqref{MainProExtScl} in the space $Y_{(-1)}^\alpha$ for $\alpha_0\leq\alpha<1$, for some $0<\alpha_0<1$ sufficiently near to 1 such that 
\[
Y_{(-1)}^1\hookrightarrow Y_{(-1)}^{\alpha_0}\hookrightarrow Y.
\]
The existence of such $\alpha_0$   is assured by the continuous embedding from fractional powers theory that asserts 
\[
Y_{(-1)}^\beta\hookrightarrow Y_{(-1)}^\gamma
\]
for $0\leq\gamma\leq\beta\leq1$.
  Initially, we collect some properties of the nonlinearities $f$ and $\mathbb{F}$. We observe that \eqref{NovaEstGGDDDima} implies that, thanks to the Mean Value Theorem,
\begin{equation}\label{NovaEstima}
|f_i(s_1)-f_i(s_2)|\leq c(1+|s_1|^{\varrho-1}+|s_2|^{\varrho-1})|s_1-s_2|        
\end{equation}
for any $s_1,s_2\in\mathbb{R}$, $i=1,2$, for some $c>0$.

\begin{lemma}\label{LemmAUXIL}
Let $f$ be a twice continuously differentiable function such that the growth conditions \eqref{CondNonLinf} and \eqref{NovaEstGGDDDima} hold. Then 
\begin{equation}\label{Conf45Tef3}
\|f_i(u_1)-f_i(u_2)\|_{L^{\frac{2N}{N+2m}}(\Omega)}\leq c(1+\|u_1\|_{H^m(\Omega)}^{\varrho-1}+\|u_2\|_{H^m(\Omega)}^{\varrho-1})\|u_1-u_2\|_{H^m(\Omega)},
\end{equation}
for any  $u_1,u_2 \in H^m(\Omega)$, and $i=1,2$.
\end{lemma}

\begin{proof}
From \eqref{NovaEstima}, the H\"{o}lder inequality, and the Sobolev embeddings we obtain
\[
\begin{split}
\|f_i(u_1)-f_i(u_2)\|_{L^{\frac{2N}{N+2m}}(\Omega)}&\leq c\left[\int_\Omega(1+|u_1|^{\varrho_i-1}+|u_2|^{\varrho_i-1})^{\frac{2N}{N+2m}}|u_1-u_2|^{\frac{2N}{N+2m}}dx\right]^{\frac{N+2m}{2N}}\\
 &\leq c|1+\|1+|u_1|^{\varrho_i-1}+|u_2|^{\varrho_i-1}\|_{L^{\frac{N}{2m}}(\Omega)}\|u_1+u_2\|_{L^{\frac{2N}{N-2m}}(\Omega)}\\
&\leq C\left(1+\|1+\|u_1\|_{L^{\frac{N(\varrho_1-1)}{2m}}(\Omega)}^{\varrho_i-1}+\|u_2\|_{L^{\frac{N(\varrho_1-1)}{2m}}(\Omega)}^{\varrho_i-1}\right)\|u_1+u_2\|_{L^{\frac{2N}{N-2m}}(\Omega)}.
\end{split}
\]
Since
\[
\dfrac{2N(\varrho_i-1)}{4m}\leq \dfrac{2N}{N-2m}
\]
we have 
\[
L^{\frac{N(\varrho_i-1)}{2m}}(\Omega)\supset H^m(\Omega).
\]
for any $i=1,2$. 

We observe that if $\varrho_i=\frac{N+2m}{N-2m}$, then $L^{\frac{N(\varrho_i-1)}{2m}}(\Omega)=L^{\frac{4mN}{N-2m}}(\Omega)$, and $(L^{\frac{4mN}{N-2m}}(\Omega))^*=L^{\frac{2N}{N+2m}}(\Omega)$.
This completes the proof. $\Box$
\end{proof}

\bigskip

This proves the following result.

\begin{lemma}
Let $f$ be a twice continuously differentiable function such that the growth conditions \eqref{CondNonLinf} and \eqref{NovaEstGGDDDima} hold.   Then  $\mathbb{F} : Y_{(-1)}^\alpha \to Y_{(-1)}$, the nonlinearity given in \eqref{NonLinF}, is Lipschitz continuous in bounded subsets of $Y_{(-1)}^\alpha$ for any $\alpha_0\leq\alpha<1$. 
\end{lemma}

\begin{proof}
Let  $\alpha_0\leq\alpha<1$. For each $\begin{bmatrix}     u_1\\ v_1\\ w_1 \end{bmatrix},\begin{bmatrix}     u_2\\ v_2\\ w_2 \end{bmatrix}\in Y_{(-1)}^\alpha$, we have
\[
\begin{split}
\left\|\mathbb{F}(\begin{bmatrix}     u_1\\ v_1\\ w_1 \end{bmatrix})-\mathbb{F}(\begin{bmatrix}     u_2\\ v_2\\ w_2 \end{bmatrix})\right\|_{Y_{(-1)}}&=\|f(u_1,v_1,w_1)-f(u_2,v_2,w_2)\|_{X^{-\frac12}}\\
&\leq\|f_1(u_1)-f_1(u_2)\|_{X^{-\frac12}}\\
&+\|f_2(v_1)-f_2(v_2)\|_{X^{-\frac12}}\\
&+\|f_3(w_1)-f_3(w_2)\|_{X^{-\frac12}}\\
&\leq c\|f_1(u_1)-f_1(u_2)\|_{L^{\frac{2N}{N+2m}}(\Omega)}\\
&+c\|f_2(v_1)-f_2(v_2)\|_{L^{\frac{2N}{N+2m}}(\Omega)}\\
&+c\|f_3(w_1)-f_3(w_2)\|_{L^2(\Omega)}\\
\end{split}
\]
for some $c>0$.

Thanks to the Lemma \ref{LemmAUXIL} we obtain
\[
\begin{split}
\left\|\mathbb{F}(\begin{bmatrix}     u_1\\ v_1\\ w_1 \end{bmatrix})-\mathbb{F}(\begin{bmatrix}     u_2\\ v_2\\ w_2 \end{bmatrix})\right\|_{Y_{(-1)}}
&\leq  c(1+\|u_1\|_{H^m(\Omega)}^{\varrho-1}+\|u_2\|_{H^m(\Omega)}^{\varrho-1})\|u_1-u_2\|_{H^m(\Omega)}\\
&+ c(1+\|v_1\|_{H^m(\Omega)}^{\varrho-1}+\|v_2\|_{H^m(\Omega)}^{\varrho-1})\|v_1-v_2\|_{H^m(\Omega)}\\
&+c\|w_1-w_2\|_{L^2(\Omega)}\\
&\leq c\left(1+\left\|\begin{bmatrix}     u_1\\ v_1\\ w_1 \end{bmatrix}\right\|^{\varrho-1}_{Y_{(-1)}^\alpha}+\left\|\begin{bmatrix}     u_2\\ v_2\\ w_2 \end{bmatrix}\right\|^{\varrho-1}_{Y_{(-1)}^\alpha}\right)\left\|\begin{bmatrix}     u_1\\ v_1\\ w_1 \end{bmatrix}-\begin{bmatrix}     u_2\\ v_2\\ w_2 \end{bmatrix}\right\|_{Y_{(-1)}^\alpha},
\end{split}
\]
for some $c>0$. $\Box$    
\end{proof}

\begin{definition}\label{NewSenDefSol}
Let $\mathbb{A}$ be the unbounded linear operator defined by \eqref{OpA001}-\eqref{OpA002}, and $\alpha_0\leq\alpha<1$. For the continuous function $\mathbb{F} : Y_{(-1)}^\alpha \to Y_{(-1)}$, the function
    $ {\bf u}(\cdot, {\bf u}_0): [0, T_0) \to Y_{(-1)}^\alpha$ is a solution for the problem
\begin{equation}\label{MainProExtScl}
\begin{cases}
\dfrac{d{\bf u}}{dt}+\mathbb{A}_{(-1)}{\bf u}=\mathbb{F}({\bf u}),\ t>0,\\
{\bf u}(0)={\bf u}_0,
\end{cases}
\end{equation}
if it is continuous, continuously differentiable in $(0, T_0), {\bf u}(t,  {\bf u}_0) \in D(\mathbb{A}_{(-1)})$ for $t \in (0, T_0)$ and \eqref{abstrcauchyp} is satisfied for all $t \in (0, T_0)$.
\end{definition}

\begin{theorem}\label{DtTheoremmm}
Let $f$ be a twice continuously differentiable function such that the growth conditions \eqref{CondNonLinf} and \eqref{NovaEstGGDDDima} hold.  Let $\mathbb{A}$ be the unbounded linear operator defined by \eqref{OpA001}-\eqref{OpA002}, and let  $\mathbb{F} : Y_{(-1)}^\alpha \to Y_{(-1)}$ the nonlinearity given in \eqref{NonLinF} for $\alpha_0\leq\alpha<1$. Then, given $r > 0$ there exists a $T_0 > 0$ and for each ${\bf u}_0 \in Y_{(-1)}^\alpha$ with $\|{\bf u}_0\|_{Y_{(-1)}^\alpha} \leq r$ a function 
\[
{\bf u}(\cdot,  {\bf u}_0) \in C([0, T_0], Y_{(-1)}^\alpha) \cap C^{1}((0, T_0], Y_{(-1)}^\alpha)
\]
with the property that
    $$\left\{{\bf u}_0 \in Y_{(-1)}^\alpha: \|{\bf u}_0\|_{Y_{(-1)}^\alpha}\ \leq r\right\} \ni z_0 \mapsto {\bf u}(\cdot,  {\bf u}_0) \in C([0,  T_0], Y_{(-1)}^\alpha)$$
is continuous; $ {\bf u}(\cdot, {\bf u}_0)$ is the unique solution of \eqref{MainProExtScl} (in the sense of Definition \ref{NewSenDefSol}) for $\alpha_0\leq\alpha<1$. If ${\bf u}_0\in X^1\times X^1\times X^1$, then ${\bf u}(\cdot)$ is twice continuously differentiable in $(0,T)$ with values in $X^{\frac12}\times X^{\frac12}\times X$. 
\end{theorem}

\begin{proof}
The proof of this theorem consists in the verification of the assumptions of Theorem \ref{existenceofsolution} and the proof of the time
regularity in the last statement of it. First we remember the sectoriality of $\mathbb{A}_{(-1)}$ using Theorem \ref{TSerFRef}.  From all this and from Theorem \ref{existenceofsolution} all assertions of Theorem \ref{DtTheoremmm} follows except the last one. Let us consider
\[
\partial^{4}_tu+\alpha \partial^{3}_tu +\beta A\partial^{2}_tu +\gamma A\partial_t u+\delta A\partial^{3}_t u=\partial_1f\partial_tu+\partial_2f\partial_t^2u+\partial_3f\partial_t^2u,
\]
where $\partial^k_{t}u$ denotes the $k$-order derivative of the function $u$ with respect to time, and $\partial_if$ denotes the partial derivative of the function $f$ with respect to $i-$th coordinate. To see that  the last statement holds let ${\bf u}_0=\begin{bmatrix}
u_0\\ v_0\\ w_0%\\ \partial^{3}_tu(0)   
\end{bmatrix}\in X^1\times X^1\times X^1 $, ${\bf u}(\cdot,{\bf u}_0)\in C^2((0,T);X^\frac12\times X^\frac12\times X )$, where ${\bf u}=\begin{bmatrix}
u\\ \partial_t u\\ \partial^{2}_t u%\\ \partial^{3}_t u    
\end{bmatrix}$, we apply again Theorem \ref{existenceofsolution} to the following Cauchy problem
\[
\dfrac{d}{dt}\begin{bmatrix}
u\\ v\\ w\\ z   
\end{bmatrix}    +\mathbb{B} \begin{bmatrix}
u\\ v\\ w\\ z   
\end{bmatrix}=\overline{F}\left(\begin{bmatrix}
u\\ v\\ w\\ z   
\end{bmatrix}\right),\quad t>0,
\]
\[
\begin{bmatrix}
u(0)\\ v(0)\\ w(0)\\ z(0)   
\end{bmatrix}=\begin{bmatrix}
u_0\\ v_0\\ w_0\\ -\alpha w_0-\beta A v_0-\gamma Au_0-\delta Aw_0+{\bf f}(u_0)   
\end{bmatrix}
\]
with $Z=X^1\times X^1\times X^1\times X$, where $\mathbb{B}$ denotes the unbounded linear operator $\mathbb{B}:D(\mathbb{B})\subset Z\to Z$ given by
\[
D(\mathbb{B})=X^1\times X^1\times X^1\times X^1
\]
such that
\[
\forall \begin{bmatrix}
u\\ v\\ w\\ z   
\end{bmatrix}\in D(\mathbb{B}),\quad \mathbb{B}\begin{bmatrix}
u\\ v\\ w\\ z   
\end{bmatrix}=\begin{bmatrix}
0 & -I & 0 & 0\\ 0 & 0 & -I & 0\\ 0 & 0 & 0 & -I\\ 0& \gamma A & \beta A & \alpha I+\delta A  
\end{bmatrix} \begin{bmatrix}
u\\ v\\ w\\ z   
\end{bmatrix}= \begin{bmatrix}
-v\\ -w\\ -z\\ \gamma A v+\beta Aw+\alpha z+\delta Az,
\end{bmatrix}  
\]
and
\[
\overline{F}\left(\begin{bmatrix}
u\\ v\\ w\\ z   
\end{bmatrix}\right)=\begin{bmatrix}
0\\ 0\\ 0\\ \overline{{\bf f}}(u,v)   
\end{bmatrix}
\]
with 
\[
\overline{{\bf f}}(u,v,w)(x)=\partial_1f((u,v,w)(x))v(x)+\partial_2f((u,v,w)(x))w(x)+\partial_3f((u,v,w)(x))\partial_tw(x). 
\]
Firstly, we need to rewrite the problem since $\mathbb{B}$ is singular and is not invertible in $Z$. Therefore, we consider the Cauchy problem
\[
\dfrac{d}{dt}\begin{bmatrix}
u\\ v\\ w\\ z   
\end{bmatrix}    +\mathbb{B}_0 \begin{bmatrix}
u\\ v\\ w\\ z   
\end{bmatrix}=\overline{H}\left(\begin{bmatrix}
u\\ v\\ w\\ z   
\end{bmatrix}\right),\quad t>0,
\]
\[
\begin{bmatrix}
u(0)\\ v(0)\\ w(0)\\ z(0)   
\end{bmatrix}=\begin{bmatrix}
u_0\\ v_0\\ w_0\\ -\alpha w_0-\beta v_0-\gamma Au_0-\delta Aw_0+{\bf f}(u_0,v_0,w_0)   
\end{bmatrix}
\]
with $Z=X^1\times X^1\times X^1\times X$, and the unbounded linear operator $\mathbb{B}_0:D(\mathbb{B}_0)\subset Z\to Z$ given by
\[
D(\mathbb{B}_0)=X^1\times X^1\times X^1\times X^1
\]
and
\[
\forall \begin{bmatrix}
u\\ v\\ w\\ z   
\end{bmatrix}\in D(\mathbb{B}_0),\quad \mathbb{B}_0\begin{bmatrix}
u\\ v\\ w\\ z   
\end{bmatrix}=\begin{bmatrix}
-I & 0 & 0 & 0\\ 0 & -I & 0 & 0\\ 0 & 0 & 0 & -I\\ 0 & \gamma A & \beta A & \alpha I+\delta A
\end{bmatrix} \begin{bmatrix}
    u \\ v \\ w \\ z
\end{bmatrix} = \begin{bmatrix}
-u\\ -v\\ -z\\ \gamma A v + \beta A w+ (\alpha I+\delta A)z,
\end{bmatrix}  
\]
and
\[
\overline{H}\left(\begin{bmatrix}
u\\ v\\ w\\ z   
\end{bmatrix}\right)=\begin{bmatrix}
v-u\\ w-v\\ 0\\ \overline{h}(u,v,w,z)   
\end{bmatrix}
\]
with $\overline{h}(u,v,w,z)(x)=\overline{{\bf f}}(u,v,w)(x)$. 

Then, there is a unique solution which is a continuously
differentiable function with values in $Z=X^1\times X^1\times X^1\times X$  since the nonlinearity $\overline{h}:Z\to Z$ is Lipschitz in bounded subsets of $Z$ and  $\mathbb{B}_0$  easily satisfies all properties of  Theorem \ref{existenceofsolution}. Also, using techniques similar to those used to study ${\bf f}$, we obtain
properties of the nonlinearity $\overline{h}:X^1\times X^1\times X^1\times X\to X$ that ensure that $\overline{H}$  also satisfies the hypothesis of  Theorem \ref{existenceofsolution} and
proves the last statement of Theorem \ref{DtTheoremmm}.
\end{proof}

\section{Statements and Declarations}

The author(s) declare(s) that there is no conflict of interest.

\medskip

\noindent{\bf Funding}

This work was supported by CNPq/Brazil (Grant number  \# 303039/2021-3).

\medskip

\noindent{\bf Competing Interests}

The authors have no relevant financial or non-financial interests to disclose.

\noindent{\bf Author Contributions}

All authors contributed to the study conception and design. 

\end{document}